\documentclass{amsart}
\usepackage{graphicx}
\usepackage{amsfonts}
\usepackage{float}
\newtheorem{tm}{Theorem}

\newtheorem{defi}{Definition}

\newtheorem{rem}{Remark}
\newtheorem{rems}{Remarks}

\newtheorem{ex}{Example}

\newtheorem{prop}{Proposition}
\newtheorem{nota}{Notation}
\newtheorem{quest}{Question}

\begin{document}
\title{Beyond Descartes' rule of signs}
\author{Vladimir Petrov Kostov}
\address{Universit\'e C\^ote d’Azur, CNRS, LJAD, France}
\email{vladimir.kostov@unice.fr}

\begin{abstract}
  We consider real univariate polynomials
  with all roots real. Such a polynomial with $c$ sign changes and $p$ sign
  preservations in the sequence of its coefficients has $c$ positive and $p$
  negative roots counted with multiplicity. Suppose that all moduli of roots
  are distinct; we consider them as ordered on the positive half-axis.
  We ask the question: If the
  positions of the sign changes are known, what can the positions of the
  moduli of negative roots be? We prove several new results which show how far
  from trivial the answer to this question is.


\end{abstract}
\maketitle

\section{Introduction}

In the present paper we study a problem related to a generalization
of Descartes' rule of signs formulated in
\cite{FoKoSh}. About this rule see \cite{Ca},
\cite{Cu}, \cite{DG}, \cite{Des}, \cite{Fo}, \cite{Ga}, 
\cite{J}, \cite{La} or~\cite{Mes}. For its tropical analog
see~\cite{FoNoSh}.

A degree $d$ real polynomial $Q:=\sum _{j=0}^da_jx^j$ is {\em hyperbolic} if
all its roots are real. Suppose that all coefficients $a_j$ are non-zero.
For such a polynomial, Descartes' rule of signs
implies that it has $c$ positive and $p$ negative roots (counted with
multiplicity, so $c+p=d$), where $c$ is the
number of sign changes and $p$ the number of sign preservations in the sequence
of coefficients of $Q$. The signs of these coefficients define the
{\em sign pattern} $({\rm sgn}(a_d)$, ${\rm sgn}(a_{d-1})$, $\ldots$,
${\rm sgn}(a_0))$. We deal mainly with monic polynomials in which case 
sign patterns begin with a~$+$. In this case we can use instead of
and equivalently to a sign pattern the corresponding 
{\em change-preservation pattern} which is a $d$-vector and (by some abuse
of notation) whose $j$th
component equals $c$ if $a_{d-j+1}a_{d-j}<0$ and $p$ if
$a_{d-j+1}a_{d-j}>0$.

One can consider also the moduli of the roots
of a hyperbolic polynomial defining a given sign pattern.
We study the generic case when all moduli are distinct. A natural question
to ask is:

\begin{quest}\label{quest1}
  When these moduli are ordered on the real positive half-axis, at 
  which positions can the moduli of the negative roots be?
  \end{quest}

Descartes' rule of signs provides no hint for the answer to this question.
In the present paper we recall known and we introduce new results in this
direction which show how far from trivial the situation is.

\begin{nota}
  {\rm (1) We denote by $0<\alpha _1<\cdots <\alpha _c$ the positive and by
    $0<\gamma _1<\cdots <\gamma _p$ the moduli of the negative roots of
    a hyperbolic polynomial. We explain the notation of the order of these
    moduli on the positive half-axis by an example. Suppose that
    $d=6$, $c=2$, $p=4$ and}
  $$\alpha _1<\gamma _1<\gamma _2<\alpha _2<\gamma _3<\gamma _4~.$$
  {\rm Then for the order of moduli we write $PNNPNN$, i.~e.
    the letters $P$ and $N$ denote the relative positions of the moduli of
    the positive and negative roots.
    \vspace{1mm}
    
    (2) A sign pattern beginning with $i_1$ signs $+$ followed by $i_2$
    signs $-$ followed by $i_3$ signs $+$ etc. is denoted by
    $\Sigma _{i_1,i_2,i_3,\ldots}$.}
\end{nota}

In what follows we consider for each given degree $d$
couples of the form (change-preservation pattern,
order of moduli) (called {\em couples} for short). Such a couple is
{\em compatible} with Descartes' rule of signs if the number of components
$c$ (resp. $p$) of the change-preservation pattern is equal to the number
of components $P$ (resp. $N$) of the order of moduli. A couple is called
{\em realizable} if there exists a polynomial defining the change-preservation
pattern of the couple and whose moduli of roots define the given order.

\begin{rem}\label{remchi}
  {\rm For fixed $d$ and $c$, there are ${d\choose c}$ change-preservation
    patterns and ${d\choose c}$ orders of moduli hence ${d\choose c}^2$
    compatible couples. Thus for a given degree $d$, the total number of
    compatible couples is}

  \begin{equation}\label{eqtotal}
    \chi (d):=\sum _{c=0}^d{d\choose c}^2=\sum _{c=0}^d{d\choose c}{d\choose d-c}=
         {2d\choose d}~.
         \end{equation}
  {\rm This is the coefficient of $x^d$ in the polynomial
    $(x+1)^d(x+1)^d=(x+1)^{2d}$. Using Stirling's formula
    $n!\sim \sqrt{2\pi n}(n/e)^n$ one concludes that
    $\chi (d)\sim 2^{2d}/\sqrt{\pi d}$.}
  \end{rem}

\begin{ex}\label{exd12}
  {\rm (1) For $d=1$, the only compatible couples are
    $(c,~P)$ and $(p,~N)$. They are realizable respectively by the polynomials
    $x-1$ and $x+1$.
    \vspace{1mm}
    
    (2) For $d=2$, there are ${4\choose 2}=6$ compatible couples.
    Out of these, the couples
    $(cp,~PN)$ and $(pc,~NP)$ are not realizable. Indeed, for a hyperbolic
    polynomial $x^2-ux-v$ (resp. $x^2+ux-v$), $u>0$, $v>0$, one has the order
    of moduli $NP$ (resp. $PN$). The remaining $4$ couples are realizable.
    To see this one can consider the family of polynomials $x^2+a_1x+a_0$. In
    the plane of the variables $(a_1,a_0)$ the domain of hyperbolic polynomials
    is the one below the parabola $\mathcal{P}:a_0=a_1^2/4$. We list the
    realizable couples and the open domains in which they are realizable:}

  $$\begin{array}{llll}
    (cc,~PP)&\{ a_1<0,~0<a_0<a_1^2/4\} ~,&
    (pp,~NN)&\{ a_1>0,~0<a_0<a_1^2/4\} ~,\\ \\
    (cp,~NP)&\{ a_1<0,~a_0<0\} ~,&(pc,~PN)&\{ a_1>0,~a_0<0\} ~.\end{array}$$
  \end{ex}

We can make Question~\ref{quest1} more precise:

\begin{quest}\label{quest2}
  For a given degree $d$, which compatible couples are realizable?
\end{quest}

The above example answers this question for $d=1$ and $2$. For
$d=3$, $4$ and $5$, the exhaustive answer is given in Section~\ref{secd345}.

\begin{rem}\label{remimir}
  {\rm There exist two commuting involutions acting on the set of
    degree $d$ polynomials with non-vanishing coefficients. These are}

  $$i_m~:~Q(x)\mapsto (-1)^dQ(-x)~~~\, \, \, {\rm and}~~~\, \, \, 
  i_r~:~Q(x)\mapsto x^dQ(1/x)/Q(0)~.$$
  {\rm The role of the factors $(-1)^d$ and $1/Q(0)$ is to preserve the set
    of monic polynomials. When acting on a couple, the involution $i_m$
    changes the components $c$ to $p$, $P$ to $N$ and vice versa while the
    involution $i_r$ reads the vectors of a given couple from the right.
    A given couple is realizable or not simultaneously with all other couples
    from its orbit under the action of $i_m$ and $i_r$. An orbit consists of
  four or two couples.}
\end{rem}

\begin{nota}\label{notaklr}
  {\rm For a sign pattern $\sigma$, we denote by $k^*(\sigma )$
    the number of orders of moduli with which $\sigma$ is realizable.
    For an order of moduli 
    $\Omega$, we denote by $l_*(\Omega )$ the number of sign patterns
    realizable with $\Omega$. For a given $d$, we denote by
    $\tilde{r}^*(d)$ the ratio between the numbers
  of realizable and of all compatible couples.}
\end{nota}

\begin{ex}
  {\rm (1) For the sign pattern $\Sigma _{3,3,1}$ one has $k^*(\Sigma _{3,3,1})=6$.
    Indeed, consider the polynomial}

  $$
(x-1)(x+1)^4(x-b)=
  x^6+(3-b)x^5+(2-3b)x^4+(-2b-2)x^3+(2b-3)x^2+(3b-1)x+b~.$$
  {\rm For $b>0$ sufficiently small, it defines the sign pattern
    $\Sigma _{3,3,1}$. One can perturb its $4$-fold root at $-1$ to obtain
    polynomials with the same sign pattern and with exactly $k$ moduli of
    negative roots which are $>1$ and $4-k$ moduli which are $<1$, where 
    $k=0$, $1$, $\ldots$, $4$; these moduli are close to~$1$. On the other
    hand the only other realizable order with this sign pattern is}

  $$\gamma _1<\alpha _1<\alpha _2<\gamma _2<\gamma _3<\gamma _4~,~~~\,
  {\rm i.~e.}~~~\, NPPNNN~,$$
  {\rm see \cite[Theorems~3 and~4]{KoPuMaDe}, which makes a total of $6$
    orders of moduli realizable with $\Sigma _{3,3,1}$.
    \vspace{1mm}
    
    (2) For $m\geq 1$, $n\geq 1$, one has $k^*(\Sigma _{m,n})=2\min (m,n)-1$,
    see~\cite[Theorem~1 and Corollary~1]{KoPuMaDe}.}
  \end{ex}

Our first result is the following theorem:

\begin{tm}\label{tm1}
  (1) For $d\geq 1$, the only orders realizable with all compatible
  change-preservation
  patterns are $PP\ldots P$ and $NN\ldots N$. The corresponding
  change-preser\-vation patterns are $cc\ldots c$ and $pp\ldots p$.
  \vspace{1mm}
  
  (2) For any $d\geq 1$,
  there exist sign patterns realizable with all compatible orders. For
  $d\geq 5$, there exist sign patterns with $c=2$ which are realizable with
  all ${d\choose 2}$ compatible orders.
  \vspace{1mm}
  
  (3) There exists no sign pattern $\sigma$ such that $k^*(\sigma )=2$.
  \vspace{1mm}
  
  (4) The only sign patterns $\sigma$ with $k^*(\sigma )=3$ are
  the ones of the form $\Sigma _{2,d-1}$, $i_r(\Sigma _{2,d-1})$,
  $i_m(\Sigma _{2,d-1})$ and $i_ri_m(\Sigma _{2,d-1})$.
  \vspace{1mm}
  
  (5) For any $\ell \in \mathbb{N}^*$,
  there exist a degree $d$ and an order $\Omega$ such that
  $l_*(\Omega )=\ell$.
  \end{tm}

The theorem is proved in Section~\ref{secprtm1}. In Section~\ref{seccanrig}
we recall some notions
and known results and we continue the formulation of the new ones.
In particular, for each of the $6$ classes of non-realizable couples introduced
in Section~\ref{seccanrig} we compare the number 
of couples which it contains
with the number of all compatible couples, see~(\ref{eqtotal}).
In all $6$ cases the limit of their ratio as $d\rightarrow \infty$ is~$0$
(see part (2) of Remarks~\ref{remscanon}, part (2) of Remarks~\ref{remsrigid},
Remark~\ref{remNPd}, Remark~\ref{remc2}, Remark~\ref{remeven} and part (4) of
  Theorem~\ref{tm2parts}).
On the other hand, when considering the cases $d=3$, $4$ and $5$ in
Section~\ref{secd345}, we arrive to the conclusion that it is plausible
to have $\lim _{d\rightarrow \infty}\tilde{r}^*(d)=0$ (see Notation~\ref{notaklr}).
This however cannot be
explained by the presence of the $6$ classes of non-realizable couples, so
for the moment it is not evident what the exhaustive answer to
Question~\ref{quest2} should~be.

We finish
this section by a result of geometric nature. Consider the space of
coefficients $Oa_{d-1}\cdots a_0\cong \mathbb{R}^d$.
The {\em hyperbolicity domain} is the set
of values of $(a_{d-1},\ldots ,a_0)$ for which the corresponding monic
polynomial $Q$ is hyperbolic. The resultant $R:={\rm Res}(Q(x),(-1)^dQ(-x),x)$ 
vanishes
exactly when $Q$ has two opposite roots or a root at $0$. When the
coefficients $a_j$ are real, the polynomials $Q(x)$ and $Q(-x)$ have a
root in common either when $Q(0)=0$ or when $Q$ has two opposite real non-zero
roots or when $Q$ has a pair of purely imaginary roots.

\begin{ex}
  {\rm For $d=1$, $2$ and $3$, one obtains $R=-2a_0$, $R=4a_0a_1^2$ and
    $R=-8a_0(a_2a_1-a_0)^2$ respectively.}
\end{ex}
We denote by $[.]$ the integer part and we set

$$\begin{array}{ll}
  Q^1:=x^{[d/2]}+a_{d-2}x^{[d/2]-1}+a_{d-4}x^{[d/2]-2}+\cdots ~,&\\ \\ 
  Q^2:=a_{d-1}x^{[(d-1)/2]}+a_{d-3}x^{[(d-1)/2]-1}+a_{d-5}x^{[(d-1)/2]-2}+\cdots &
  {\rm and}\\ \\ R_0:={\rm Res}(Q^1(x),Q^2(x),x))~.&  
  \end{array}$$

\begin{tm}\label{tmres}
  (1)  One has $R=(-1)^{[d/2]+1}2^{d-[(d+1)/2]+1}a_0R_0^2$.
  \vspace{1mm}
  
  (2) The quantity $R_0$ is an irreducible polynomial
  in the variables $a_j$.
  \end{tm}

The theorem is proved in Section~\ref{secprtmres}. Properties of
the set $\{ R_0=0\}$ and its pictures for $d\leq 4$
can be found in~\cite{GaKoTa}.

\section{Canonical sign patterns, rigid orders of moduli and further results
  \protect\label{seccanrig}}

\begin{defi}
  {\rm For a given change-preservation pattern, the corresponding
    {\em canonical order} is obtained by reading the pattern from the right
    and by replacing each component $c$ (resp. $p$) by $P$ (resp. by $N$).
    E.~g., the canonical order corresponding to the pattern $ccpcp$ is
    $NPNPP$. This definition allows to define the canonical order corresponding
  to each given sign pattern beginning with~$+$.}
\end{defi}

Each sign or change-preservation pattern is realizable with its canonical order,
see \cite[Proposition~1]{KoSe}.

\begin{defi}\label{deficanonrigid}
  {\rm (1) A sign pattern (or equivalently a change-preservation pattern)
    realizable only with its corresponding
    canonical order is called {\em canonical}.
\vspace{1mm}

(2) If all monic hyperbolic polynomials having a given order of moduli
define one and the
same sign pattern, then the order is called {\em rigid}.}
\end{defi}

\begin{rems}\label{remscanon}
  {\rm (1) It is shown in \cite{KoRM} that canonical are exactly these sign
    patterns which
have no four consecutive signs equal to}

$$(+,+,-,-,)~,~~~\, (-,-,+,+)~,~~~\, (+,-,-,+)~~~\, {\rm or}~~~\, (-,+,+,-)~.$$
{\rm Hence canonical are these change-preservation patterns having no isolated
sign changes and no isolated sign preservations, i.~e. having no three
consecutive components $cpc$ or~$pcp$.

(2) In the proof of Proposition~10 in \cite{KoRM} the set of {\em all}
canonical change-preservation patterns is represented as union of four
subsets, namely of 
patterns beginning with a single $p$ or $c$, patterns
ending by a single $p$ or $c$,
patterns both beginning and ending by a single $p$ or $c$ and patterns whose
two first letters are equal and whose last two letters are also equal.
For $d\geq 100$, the
number of patterns in each of these sets can be majorized by
$2\cdot [d/2]\cdot 2^{d-[0.26d]-1}$. Hence the number of {\em all} canonical
sign-preservation patterns is $\leq \tau (d):=8\cdot [d/2]\cdot 2^{d-[0.26d]-1}$
and for large $d$, 
the number of all non-realizable couples with canonical sign-preservation
patterns is}

$$\leq \tau (d)\sum _{c=0}^d{d\choose c}=8\cdot [d/2]\cdot 2^{2d-[0.26d]-1}<
2^{2d}/\sqrt{\pi d}\sim \chi (d)~,$$
{\rm see Remark~\ref{remchi}; we majorize one of the factors ${d\choose c}$
  in (\ref{eqtotal}) by $\tau (d)$.}
\end{rems}

\begin{rems}\label{remsrigid}
  {\rm (1) It is proved in \cite{KoPMD22} that rigid are the orders of moduli
$PP\ldots P$, $NN\ldots N$ (defining the change-preservation patterns
    $cc\ldots c$ and $pp\ldots p$, the two corresponding couples are
    realizable by any polynomials having
distinct positive or distinct negative roots) and also}

\begin{equation}\label{eqPN}
  P_N:=PNPNPN\ldots ~,~~~\, N_P:=NPNPNP\ldots ~.
\end{equation}
{\rm Each of the
latter two orders (we call them {\em standard}) defines,
depending on the parity of $d$, one of the
sign patterns}

\begin{equation}\label{eqsigmapm}
  \sigma _+:=(+,+,-,-,+,+,-,-,\ldots )~~~\, {\rm or}~~~\,
  \sigma _-:=(+,-,-,+,+,-,-,+,+,\ldots )~.
\end{equation}

{\rm (2) For each fixed degree $d$, there are ${d\choose [d/2]}$ compatible
  couples with the order $P_N$ and ${d\choose [d/2]}$ with the order $N_P$,
  see~(\ref{eqPN}).
  Hence there are $2{d\choose [d/2]}-2$ compatible couples in which the order
  of moduli is rigid (more exactly standard) and which are not realizable,
  and one has 
  $\lim _{d\rightarrow \infty}(2{d\choose [d/2]}-2)/\chi (d)=0$,
  see~(\ref{eqtotal}) and use Stirling's formula.} 
\end{rems}

\begin{defi}
  {\rm We call {\em superposition} 
    of two standard orders of moduli
    $\Omega _1$ and $\Omega _2$ any order obtained as follows. One inserts the
    components of $\Omega _2$ at any places between the components of
    $\Omega _1$ or in front of the first or after the last component of
    $\Omega _1$ by preserving their relative order. Example: the order}

  $$P\bar{N}NP\bar{P}N\bar{N}\bar{P}\bar{N}~~~\, {\rm is~superposition~of}~~~\,
  PNPN~~~\, {\rm and}~~~\,
    NPNPN$$
    {\rm (we overline in this superposition the moduli coming from
      $\Omega _2$; 
      in this example there is more than one way to attribute the moduli of
      roots in the superposition as coming from $\Omega _1$ or
      $\Omega _2$; the superposition of two standard orders is not uniquely
      defined).}
\end{defi}

The following proposition explains how one can obtain new examples of
non-realizable couples on the basis of standard orders.

\begin{prop}\label{propsuper}
  Each superposition
  of two standard orders is realizable only with sign patterns of the
  form

  $$(+,+,?,-,?,+,?,-,\ldots)~,~~~\, (+,?,-,?,+,?,-,\ldots )~~~\,
  {\rm or}~~~\, (+,-,?,+,?,-,?,+,\ldots )$$
  which are the ``products'' of sign patterns $\sigma _+\sigma _+$,
  $\sigma _+\sigma _-$ and $\sigma _-\sigma _-$.
\end{prop}

\begin{proof}
  Indeed, suppose that in the superposition of standard orders,
the roots coming from the order $\Omega _i$
are roots of a polynomial $T_i$, $i=1$, $2$. Then in the product $T_1T_2$
every second coefficient, the leading coefficient and the constant term are
sums of products of a coefficient of $T_1$ and a coefficient of $T_2$ either
all with opposite or all with same signs, so the corresponding
components of the ``products'' of sign patterns are well-defined.
\end{proof}

\begin{rem}\label{remNPd}
  {\rm The number of letters $N$ in a standard order is equal to the number
    of letters $P$ or differs from the latter by~$1$. Hence in the
    superposition of two standard orders the modulus of
    this difference is majorized by~$2$. Besides, not more than $[d/2]$ of
    the signs of coefficients are not determined by the order of moduli, so 
    the number of non-realizable couples corresponding to
    superpositions of standard orders is less than}

  $$2\left( {d\choose [d/2]}+{d\choose [d/2]-1}+{d\choose [d/2]-2}\right)
  \cdot 2^{[d/2]}<
  6{d\choose [d/2]}\cdot 2^{(d+1)/2}$$
  which is $\sim 12\cdot 2^{3d/2}/\sqrt{\pi d}$
  {\rm (we use Stirling's formula here). At the same time
    $\chi (d)\sim 2^{2d}/\sqrt{\pi d}$
    (see Remark~\ref{remchi}).}
  \end{rem}

There exist other situations in which the order of moduli defines the signs
of part of the coefficients of the polynomial.

\begin{ex}\label{exc2}
  {\rm Consider for $d=8k+2$, $k\in \mathbb{N}^*$, and for $c=2$
    the order of moduli}

  $$\Omega ~:~\gamma _1<\cdots <\gamma _{4k}<\alpha _1<\alpha _2<
  \gamma _{4k+1}<\cdots <\gamma _{8k}~.$$
  {\rm It is realizable only with sign patterns having two sign changes.
    Denote by $U_1$ and $U_2$ monic hyperbolic degree $4k+1$
    polynomials with roots

    $$\begin{array}{ll}
      -\gamma _1~,~~~\, -\gamma _2~,\ldots ~,~~~\, -\gamma _{2k}~,~~~\,
      -\gamma _{4k+1}~,~~~\, -\gamma _{4k+2}~,~\ldots ~,~~~\, 
      -\gamma _{6k}~,~~~\, \alpha _1&{\rm and}\\ \\
      -\gamma _{2k+1}~,~~~\, -\gamma _{2k+2}~,~\ldots ~,~~~\, -\gamma _{4k}~,~~~\, 
      -\gamma _{6k+1}~,~~~\, -\gamma _{6k+2}~,~\ldots ~,~~~\, -\gamma _{8k}~,~~~\,
      \alpha _2&\end{array}$$
    respectively. Hence they define sign patterns of the form
    $\Sigma _{m_i,n_i}$, $i=1$,~$2$.
    According to \cite[Theorem~1]{KoPuMaDe}, if $n_i<m_i$, then
    the polynomial $U_i$ has $\leq 2n_i-2$ moduli of negative roots which are
    $\leq \alpha _i$; if $n_i>m_i$, then it has $\leq 2m_i-2$ moduli of
    negative roots which are $\geq \alpha _i$. Hence one has $n_i\geq k+1$ and
    $m_i\geq k+1$. This implies that the first $k+1$ and the last $k+1$
    coefficients of the product $U_1U_2$ are positive, i.~e. the order of
    moduli $\Omega$ is not realizable with sign patterns $\Sigma _{j_1,j_2,j_3}$
    which do not satisfy the conditions $j_1\geq k+1$ and $j_3\geq k+1$.}
  \end{ex}
    
\begin{rem}\label{remc2}
  {\rm There are ${d\choose 2}^2$ compatible couples with $c=2$ hence less
    than ${d\choose 2}^2$ non-realizable couples concerned by 
    Example~\ref{exc2}. Using the involution $i_m$ (see Remark~\ref{remimir})
    one can give as many such
    examples with $c=d-2$. One has
    $\lim _{d\rightarrow \infty}{d\choose 2}^2/\chi (d)=0$,
    see~(\ref{eqtotal}).}
  \end{rem}

The proposition and theorem that follow describe other situations in which
certain compatible couples are not realizable.

\begin{prop}\label{propeven}
  Suppose that $d$ is even, that the leading monomial and the constant term are
  positive (hence $c$ is even), that all coefficients of odd powers are
  negative and that
  $c<d$. Then there is no modulus of a negative root in any of the intervals
  $(0,\alpha _1)$, $(\alpha _2,\alpha _3)$, $\ldots$,
  $(\alpha _{c-2},\alpha _{c-1})$, $(\alpha _c,\infty )$.
\end{prop}

\begin{proof}
  Indeed, for a monic hyperbolic polynomial $Q$ satisfying these conditions
  one has
  $Q(t)>0$, if $t$ belongs to any of the mentioned intervals. As all odd
  monomials are with negative coefficients, one has also $Q(-t)>Q(t)$ from
  which the proposition follows.
\end{proof}

\begin{rem}\label{remeven}
  {\rm For $d$ even, the number of sign patterns as defined in
    Proposition~\ref{propeven} is $\leq 2^{d/2}$
    (half of the signs of coefficients are fixed), so if $d$ is large, then
    the number of such non-realizable couples is}

  $$\leq 2^{d/2}\sum _{c=0}^d{d\choose c}=2^{3d/2}<\chi (d)
    \sim 2^{2d}/\sqrt{\pi d}~,$$
    {\rm see Remark~\ref{remchi}.}
    \end{rem}

\begin{tm}\label{tm2parts}
  (1) Suppose that

  \begin{equation}\label{equalphagamma}
    c\leq p~~~\, {\rm and}~~~\,  \alpha _c<\gamma _p,~~~\,
    \alpha _{c-1}<\gamma _{p-1}~,~
    \ldots ~,~~~\, \alpha _1<\gamma _{p-c+1}~.
  \end{equation}
  Then $a_{d-1}>0$. Hence a couple with $a_{d-1}<0$ and order satisfying
  conditions (\ref{equalphagamma}) is not realizable. 
  \vspace{1mm}

  (2) For fixed $d$, the number of orders of moduli satisfying
  conditions (\ref{equalphagamma}) is

  \begin{equation}\label{equalphagamma1}
    T_d^c:={d\choose c}-C_0{d-1\choose c-1}-C_1{d-3\choose c-2}-
    C_2{d-5\choose c-3}-
    C_3{d-7\choose c-4}-\cdots ~, 
  \end{equation}
  where $C_k:={2k\choose k}/(k+1)$ is the $k$th Catalan number.
  \vspace{1mm}

  (3) One has

   \begin{equation}\label{eqT}
     T_d^c={d\choose c}\left( 1-\frac{c}{d-c+1}\right) =
     {d\choose c}\frac{d-2c+1}{d-c+1}~.
   \end{equation}

   (4) For the number $\nu (d)$ of non-realizable couples satisfying
   condition (\ref{equalphagamma}) and with $a_{d-1}<0$ one has
   $\lim _{d\rightarrow \infty}\nu (d)/\chi (d)=0$, see~(\ref{eqtotal}).
\end{tm}

\begin{rem}
  {\rm The quantity $T_d^c{d-1\choose c}$ (resp. ${d\choose c}{d-1\choose c}$)
    is the number of couples in which the change-preservation pattern begins
    with $p$ and the order satisfies condition (\ref{equalphagamma}) (resp.
    of all compatible couples in which the change-preservation pattern begins
    with $p$). For $c$ fixed, one has
    $\lim _{d\rightarrow \infty}T_d^c/{d\choose c}=1$. Indeed, this is the ratio of
    two degree $c$ polynomials in $d$ whose leading coefficients
    equal $1/c!$.}  
  \end{rem}

\begin{proof}[Proof of Theorem~\ref{tm2parts}]
  Part (1). Indeed,
  $a_{d-1}=\gamma _1+\cdots +\gamma _p-\alpha _1-\cdots -\alpha _c>0$.
  \vspace{1mm}

  Part (2). The first term in the right-hand side of
  (\ref{equalphagamma1}) is the number of all
  orders with $c$ components equal to $P$. The second term is the number of
  orders beginning with $P$; they do not satisfy conditions
  (\ref{equalphagamma}). The third (resp. the fourth) term is the number
  of orders beginning with $NPP$ (resp. with $NPNPP$ or $NNPPP$). The fifth
  term is the number of orders beginning with $NPNPNPP$, $NNPPNPP$, $NPNNPPP$,
  $NNPNPPP$ or $NNNPPPP$ etc.

  That is, for $k\geq 2$,
  the $k$th term is the number of
  orders among whose first $2k-1$ components there are $k$ letters $P$ and
  which are not included in one of the previous terms (excluding the initial
  ${d\choose c}$). In an equivalent way, the $k$th term contains orders among
  whose $2k-2$ first components there are exactly $k-1$ letters $P$ and for
  $s\leq 2k-2$, among their $s$ first letters there are not less letters $N$
  than letters $P$. Hence this is the number of lattice paths in the plane
  with possible steps $(1,1)$ and $(1,-1)$ going from $(0,0)$ to $(2k-2,0)$
  which do not descend below the abscissa-axis. The number of such paths
  is~$C_{k-1}$.
  \vspace{1mm}
  
  Part (3). Formula (\ref{eqT}) can be proved by
induction on $d$. For $d=1$ and $2$ and for $c\leq d$, it is to be checked
directly. Suppose that it is true for $d\leq d_0$. Then for $d=d_0+1$, one
applies to any binomial coefficient in the formula the well-known
equality ${n\choose k}={n-1\choose k-1}+{n-1\choose k}$. Thus

$$\begin{array}{rcl}
  T_d^c=T_{d-1}^c+T_{d-1}^{c-1}&=&{d-1\choose c}\left( 1-\frac{c}{d-c}\right) +
  {d-1\choose c-1}\left( 1-\frac{c-1}{d-c+1}\right) \\ \\ &=&
  {d\choose c}\left( 1-\frac{c}{d-c+1}\right) ~,
  \end{array}$$
where the rightmost equality is to be checked straightforwardly.
\vspace{1mm}

Part (4). Suppose that $d=2k$, $k\in \mathbb{N}^*$. Set

$$h_{k,m}:=\frac{k(k-1)\cdots (k-m+1)}{(k+1)(k+2)\cdots (k+m)}~,~~~\,
{\rm so}~~~\,
{2k\choose k-m}={2k\choose k}h_{k,m}~.$$
For $k$ fixed, the sequence $h_{k,m}$ is decreasing in $m$; one has $h_{k,0}=1$.
The sum $\sum _{c=0}^d{d\choose c}^2$ of all compatible couples equals
$\tilde{b}:={2k\choose k}^2(1+2\sum _{m=1}^{k}h_{k,m}^2)$. The number
$\nu (d)=\nu (2k)$ is bounded by

$$\sum _{c=0}^k{2k\choose c}T_{2k}^c=
\sum _{m=0}^k{2k\choose k-m}T_{2k}^{k-m}={2k\choose k}^2\sum _{m=0}^k
\frac{2m+1}{k+m+1}h_{k,m}^2$$
(we remind that the orders
satisfying condition (\ref{equalphagamma}) are defined under the assumption
that $c\leq p$). Fix $s\in (0,1)$. Then

  $$g_1:=\sum _{m=0}^{[sk]}\frac{2m+1}{k+m+1}h_{k,m}^2\leq \frac{2[sk]+1}{k+[sk]+1}
    \sum _{m=0}^{[sk]}h_{k,m}^2~.$$
It is clear that $g_1<\frac{2[sk]+1}{k+[sk]+1}\sum _{m=0}^{k}h_{k,m}^2$, so

\begin{equation}\label{equs1}
{2k\choose k}^2g_1<\frac{2[sk]+1}{k+[sk]+1}\tilde{b}~.
  \end{equation}
For large values of $k$ and for $m\geq [sk]+1$, the quantity $h_{k,m}$
is majorized by

$$\frac{(k-[sk/2])\cdots (k-m+1)}{(k+[sk/2]+1)\cdots (k+m)}\leq
\left( \frac{k-[sk/2]}{k+[sk/2]+1}\right) ^{[sk]-[sk/2]}
\left( \frac{k-[sk]+1}{k+[sk]}\right) ^{m-[sk]-1}~.$$
Set $u:=\frac{k-[sk/2]}{k+[sk/2]+1}$ and $v:=\frac{k-[sk]+1}{k+[sk]}$.
Hence

 $$ \begin{array}{cclcl}
    g_2&:=&\sum _{m=[sk]+1}^{k}h_{k,m}^2&<&
    u^{[sk]-[sk/2]}\sum _{m=[sk]+1}^{\infty}v^{m-[sk]-1}\\ \\
    &=&\frac{u^{[sk]-[sk/2]}}{1-v} 
    &=&u^{[sk]-[sk/2]}\frac{k+[sk]}{2[sk]+1}~.\end{array}$$
The latter quantity tends to $0$ as $k\rightarrow \infty$, therefore
$\lim _{k\rightarrow \infty}{2k\choose k}^2g_2/\tilde{b}=0$. As
$g_3:=\sum _{m=[sk]+1}^{k}\frac{2m+1}{k+m+1}h_{k,m}^2<g_2$, one obtains

\begin{equation}\label{equs3}
\lim _{k\rightarrow \infty}{2k\choose k}^2g_3/\tilde{b}=0~.
\end{equation}
One has $\nu (d)\leq {2k\choose k}^2(g_1+g_3)$.
The coefficient of $\tilde{b}$ in (\ref{equs1}) can be made smaller than
any positive number by choosing $s$ small enough. Therefore inequality
(\ref{equs1}) and equality (\ref{equs3}) imply part (4) of
Theorem~\ref{tm2parts}
for $d$ even.
\vspace{1mm}

If $d=2k+1$, $k\in \mathbb{N}^*$, then one can prove part (4) in much the
same way, so we point out only some technical differences. One sets

$$h_{k,m}:=\frac{k(k-1)\cdots (k-m+1)}{(k+2)(k+3)\cdots (k+m+1)}~,~~~\, {\rm so}
~~~\, {2k+1\choose k-m}={2k+1\choose k}h_{k,m}~,$$
and $\tilde{b}=2{2k+1\choose k}^2(1+\sum _{m=1}^kh_{k,m}^2)$. The definitions of
the quantities $g_1$, $g_2$ and $g_3$ are the same, but with respect to the
new formula for $h_{k,m}$. One sets
$u:=\frac{k-[sk/2]}{k+[sk/2]+2}$ and $v:=\frac{k-[sk]+1}{k+[sk]+1}$.
Inequality (\ref{equs1}) and equality (\ref{equs3}) remain the same.
\end{proof}

\section{Realizable couples for $d=3$, $4$ and $5$\protect\label{secd345}}



We give the exhaustive answer to Question~\ref{quest2} for $d=3$, $4$ and
$5$; for $d=1$ and~$2$, this answer is given by Example~\ref{exd12};
one finds that $\tilde{r}^*(1)=1$ and $\tilde{r}^*(2)=2/3$, see
Notation~\ref{notaklr}.
It is clear from part (1) of Theorem~\ref{tm1} that $\tilde{r}^*(1)<1$
for $d>1$. We make use
of the involution $i_m$, see Remark~\ref{remimir}, to consider only the cases
with $a_{d-1}>0$. For $d=3$, we give the list of sign patterns and
(non)-realizable orders
in the following table:

$$\begin{array}{ccc}{\rm sign~pattern}&{\rm realizable~orders}&
  {\rm non-realizable~orders}\\ \\ (+,+,+,-)&PNN&NPN~,~NNP\\ \\
  (+,+,-,-)&PNN~,~NPN~,~NNP&\\ \\
  (+,+,+,+)&NNN&\\ \\ (+,+,-,+)&PPN&NPP~,~PNP~.\end{array}$$
Thus $\tilde{r}^*(3)=3/5$. The (non)-realizability of these cases can be
justified using the results in~\cite{KoPuMaDe}. For $d=4$,
we list the sign patterns by the value of~$c$:

$$\begin{array}{cccc}c&{\rm sign~pattern}&{\rm realizable~orders}&
  {\rm non-realizable~orders}\\ \\
  0&(+,+,+,+,+)&NNNN&\\ \\ 1&(+,+,+,+,-)&PNNN&NPNN,~NNPN,~NNNP\\ \\ 
  &(+,+,+,-,-)&PNNN,~NPNN,~NNPN&NNNP\\ \\ 
  &(+,+,-,-,-)&NPNN,~NNPN,~NNNP&PNNN\\ \\
  2&(+,+,-,+,+)&NPPN&NNPP,~NPNP,~PNNP\\
  &&&PNPN~,~PPNN\\ \\
  &(+,+,-,-,+)&PNPN,~NPPN,&NPNP,~NNPP\\ &&
  PPNN,~PNNP&\\ \\
  &(+,+,+,-,+)&PPNN&PNPN,~NPPN,~NPNP\\
  &&&PNNP,~NNPP\\ \\ 
3&(+,+,-,+,-)&PPPN&NPPP,~PNPP,~PPNP
\end{array}$$
Hence $\tilde{r}^*(4)=3/7$. The (non)-realizability of the cases can be
proved using the results
in~\cite{KoPuMaDe}. The involution $i_m$ transforms the sign pattern with $c=3$
into $(+,-,-,-,-)$. We illustrate the realizability of 
the cases with the sign pattern $(+,+,-,-,+)$ by examples:

$$\begin{array}{ll}
  PNPN&(x+1.3)(x-1.2)(x+1.1)(x-1)=\\ &x^4+0.2x^3-2.65x^2-0.266x+1.716\\ \\ 
  NPPN&(x+2)(x-1)(x-0.9)(x+0.8)=\\ &x^4+0.9x^3-2.82x^2-0.52x+1.44\\ \\
  PPNN&(x+2)(x+1.1)(x-1)(x-0.1)=\\ &x^4+2x^3-1.11x^2-2.11x+0.22\\ \\
  PNNP&(x-2)(x+1.9)(x+1)(x-0.8)=\\ &x^4+0.1x^3-4.62x^2-0.68x+3.04~.
\end{array}$$
For $d=5$, we show for each sign pattern only 
the number of realizable and the total number of orders compatible with
the sign pattern and in some cases the realizable orders. To justify the
tables below one can use the results in \cite{KoPuMaDe} and~\cite{KoRM}.
There are the following canonical sign patterns:

$$\begin{array}{lllllll}
  c=0&(+,+,+,+,+,+)&1/1&&c=1&(+,+,+,+,+,-)&1/5\\ \\
  c=2&(+,+,-,+,+,+)&1/10&&c=3&(+,+,-,+,-,-)&1/10\\
  &(+,+,+,-,+,+)&1/10&&&(+,+,+,-,+,-)&1/10\\ &(+,+,+,+,-,+)&1/10\\ \\
  c=4&(+,+,-,+,-,+)&1/5
\end{array}
$$
The remaining sign patterns are:

$$\begin{array}{llcl}c=1&(+,+,+,+,-,-)&PNNNN~,&3/5\\
  &&NPNNN~,~NNPNN&\\ &(+,+,+,-,-,-)&&5/5\\
  &(+,+,-,-,-,-)&NNPNN~,&3/5\\ &&NNNPN~,~NNNNP& \\ \\ 
  c=2&(+,+,-,-,-,+)&PPNNN~,&5/10\\ &&PNPNN~,~PNNPN~,&\\
  &&PNNNP~,~NPPNN&\\ &(+,+,+,-,-,+)&PPNNN~,~PNPNN~,&4/10\\
  &&PNNPN~,~NPPNN&\\ &(+,+,-,-,+,+)&&10/10\\ \\
  c=3&(+,+,-,+,+,-)&&5/10\\
  &(+,+,-,-,+,-)&&4/10
\end{array}$$
Therefore $\tilde{r}^*(5)=47/126$. The two latter sign patterns (with $c=3$)
are obtained from two of the sign
patterns with $c=2$ via the involution $i_mi_r$.
%
The realizability of the sign pattern $(+,+,-,-,+,+)$ with all possible orders
results from

$$(x+1)^3(x-1)^2=x^5+x^4-2x^3-2x^2+x+1~.$$
Indeed, by perturbing the triple root at $-1$ and the double root at $1$ one
obtains polynomials with the same sign pattern and with any order of the
moduli of the roots, see the proof of part (2) of Theorem~\ref{tm1}.

\begin{rem}\label{remr}
  {\rm We obtained the following sequence for the values of the quantity
$\tilde{r}^*(d)$: $1$, $2/3$, $3/5$, $3/7$, $47/126$, $\ldots$. One could
    conjecture that the sequence is decreasing. For the sequence of the ratios
    of two consecutive terms one gets}

  $$2/3=0.66\ldots ~,~~~\, 9/10=0.9~,~~~\, 5/7=0.71\ldots ~,~~~\,
  47/54=0.87\ldots ~.$$
  {\rm It seems that the even and the odd terms form two adjacent sequences
    and that $\lim _{d\rightarrow \infty}\tilde{r}^*(d)=0^+$.}
  \end{rem}

\section{Proof of Theorem~\protect\ref{tm1}\protect\label{secprtm1}}

Part (1). As already mentioned, for the orders $PP\ldots P$ and
$NN\ldots N$, the only
change-preservation patterns compatible with them are $cc\ldots c$ and
$pp\ldots p$ respectively and the corresponding couples are realizable.

Suppose that for given $c>0$ and $p>0$, the order of moduli $\Omega$ is
realizable with all compatible change-preservation patterns. Then,
in particular, it is realizable with the sign patterns
$\sigma '$ and $\sigma ''$,
    where $\sigma '$ has all its $c$ sign changes at the
    beginning followed by its $p$ sign preservations and vice-versa for
    $\sigma ''$. However, the sign patterns $\sigma '$ and $\sigma ''$ are
    canonical hence realizable only with their respective canonical orders
    $\Omega '$ and $\Omega ''$, see Definition~\ref{deficanonrigid}.
    As $\Omega '\neq \Omega ''$, the order
    $\Omega$ is not realizable with both $\sigma '$ and $\sigma ''$.
    \vspace{1mm}
    
    Part (2). For $d\geq 1$, the all-pluses sign pattern is realizable with
    its only compatible order $N\ldots N$.
    To prove the rest of part (2) for $d\geq 5$
    we construct sign patterns with $c=2$ which are
    realizable with all compatible orders. Consider the polynomial

  $$\begin{array}{ccl}
    (x+1)^{d-2}(x-1)^2&=&\left( \sum _{k=0}^{d-2}{d-2\choose k}x^k\right)
    (x^2-2x+1)\\ \\
    &=&\sum _{k=0}^{d}h_kx^k~,~~~\,
    h_k:={d-2\choose k}-2 {d-2\choose k-1}+
    {d-2\choose k-2}~.\end{array}$$
    It has two sign changes (so its sign pattern is of the form
    $\Sigma _{i_1,i_2,i_3}$). To understand in which positions they are 
  one observes that

  $$h_k=\frac{(d-2)!}{k!(d-k)!}(4k^2-4dk+d(d-1))~,$$
  so $h_k=0$ if and only if $k=k_{\pm}:=(d\pm \sqrt{d})/2$. If $d$ is not an
  exact square, then the sign changes occur between the powers $x^{s_{\pm}}$ and
  $x^{s_{\pm}+1}$, where $s_{\pm}<k_{\pm}<s_{\pm}+1$. If $d$ is an exact square, then
  the coefficients of $x^{k_{\pm}}$ are~$0$. 

  Suppose that $d$ is not an exact square. One can perturb the roots of the
  polynomial by keeping the sign
  pattern the same. If $d$ is an exact square, then one can perturb them so that
  all coefficients become non-zero. One can choose such a perturbation for any
  possible order of the moduli of roots which proves part~(2). One can
  observe that as $k_+-k_-=\sqrt{d}$, for
  $d\geq 5$, there are at least two consecutive negative coefficients (i.~e.
  $i_2\geq 2$) and
  the sign pattern is not canonical.
  \vspace{1mm}

  We prove part (3) of the theorem by induction on $d$.
  For $d=1$, $2$ and $3$, the claim is to be checked straightforwardly,
  see Example~\ref{exd12} and Section~\ref{secd345}.
  Suppose that $d\geq 4$ and that $\sigma$ is not canonical.
  Represent $\sigma$ in the form $(\sigma _d,\sigma ^{\dagger},\sigma _0)$,
  where $\sigma _d$ and $\sigma _0$ are its first and last components.
  Then at least one
  of the sign patterns $(\sigma _d,\sigma ^{\dagger})$ and
  $(\sigma ^{\dagger},\sigma _0)$ contains an isolated sign change or an
  isolated sign
  preservation. Suppose that this is $(\sigma _d,\sigma ^{\dagger})$. Then
  $(\sigma _d,\sigma ^{\dagger})$ is not canonical
  and hence is realizable by at least
  three orders by polynomials $P_j$. This means that $\sigma$ is also realizable
  by at least three orders defined by the roots of the polynomials
  $P_j(x)(x\pm \varepsilon )$, where $\varepsilon >0$ is small enough and the
  sign is~$+$ (resp.~$-$) if the last two components of $\sigma$ are equal
  (resp. are different).
  \vspace{1mm}
  
  Part (4) is also proved by induction on $d$. For $d\leq 4$,
  it is to be checked directly. Suppose that $d\geq 5$.
  If neither of the sign patterns
  $(\sigma _d,\sigma ^{\dagger})$ and
  $(\sigma ^{\dagger},\sigma _0)$ contains an isolated sign change or
  sign preservation, then this is the case of $\sigma$ as well, so $\sigma$ is
  canonical and $k^*(\sigma )=1$ -- a contradiction. Hence at least one
  of these sign patterns is not canonical. Without loss of generality we
  suppose that this is $(\sigma _d,\sigma ^{\dagger})$ (otherwise we apply the
  involution~$i_r$). Hence $k^*((\sigma _d,\sigma ^{\dagger}))\geq 3$,
  so $k^*((\sigma _d,\sigma ^{\dagger}))=3$, otherwise similarly to 
  the proof of part (3) we obtain that $k^*(\sigma )>3$.
  Applying if necessary the involution $i_m$ we assume
  that $(\sigma _d,\sigma ^{\dagger})=\Sigma _{2,d-2}$ or $\Sigma _{d-2,2}$.
  In the first case one has $\sigma =\Sigma _{2,d-1}$. Indeed, if
  $\sigma =\Sigma _{2,d-2,1}$, then $k^*(\sigma )>3$, see
  \cite[Theorems 3 and~4]{KoPuMaDe}. In the second case
  either $\sigma =\Sigma _{d-2,3}$ and $k^*(\sigma )=5$ (see
  \cite[Theorem 1]{KoPuMaDe}) or
  $\sigma =\Sigma _{d-2,2,1}$ and $k^*(\sigma )=4$ (see
  \cite[Theorems 3 and~4]{KoPuMaDe}).
  \vspace{1mm}

  Part (5). For $d$ even, the order
    $\Omega :=PNN\ldots N$ is realizable exactly with the sign patterns
  $\Sigma _{m,n}$, $m+n=d+1$, $n<m$, see \cite[Theorem~1]{KoPuMaDe},
  so $\ell _*(\Omega )=d/2$.

\section{Proof of theorem~\protect\ref{tmres}\protect\label{secprtmres}}

\begin{proof}[Proof of part (1)]

   A) For a vector-row $v$ of length $2d$ we denote by $v_{\ell}$ the vector-row
  obtained from $v$ by shifting $v$ by $\ell$ positions to the right (the
  rightmost $\ell$ positions are then lost and the leftmost $\ell$ positions
  are filled with zeros). We represent $R$ as determinant of the Sylvester
  $2d\times 2d$-martix 
  of the polynomials $Q(x)$ and $(-1)^dQ(-x)$ whose first and $(d+1)$st
  row equal respectively

  $$\begin{array}{lrrrrrrrrrr}
    u:=(~1&a_{d-1}&a_{d-2}&a_{d-3}&a_{d-4}&\ldots &a_1&a_0&0&\ldots &0~)\\ \\ 
    {\rm and}&&&&&&&&&&\\ \\ 
    w:=(~1&-a_{d-1}&a_{d-2}&-a_{d-3}&a_{d-4}&\ldots &(-1)^{d-1}a_1&(-1)^da_0&0&
    \ldots &
    0~)~;\end{array}$$
  its second and $(d+2)$nd rows equal $u_1$ and $w_1$, its third and $(d+3)$rd
  rows equal $u_2$ and $w_2$ etc. For $d=2$ and $d=3$, we obtain the
  determinants

  $$\left| \begin{array}{rrrr}
    1&a_1&a_0&0\\ 0&1&a_1&a_0\\ 1&-a_1&a_0&0\\ 0&1&-a_1&a_0
  \end{array}\right|~~~\, {\rm and}~~~\, \left| \begin{array}{rrrrrr}
    1&a_2&a_1&a_0&0&0\\ 0&1&a_2&a_1&a_0&0\\ 0&0&1&a_2&a_1&a_0\\
    1&-a_2&a_1&-a_0&0&0\\ 0&1&-a_2&a_1&-a_0&0\\ 0&0&1&-a_2&a_1&-a_0
    \end{array}\right| ~.$$
    \vspace{1mm}
  
  B) For $j=1$, $\ldots$, $d$, we add the $(j+d)$th row to the $j$th row. Hence
  the first row of the determinant is now

  $$\begin{array}{lrrrrrrrrrr}
    g:=(~2&0&2a_{d-2}&0&2a_{d-4}&\ldots &2a_{d-2[d/2]}&0&0&\ldots &0~)
  \end{array}$$
and the next $d-1$ rows equal $g_j$, $j=1$, $\ldots$, $d-1$. 
  After this one
  subtracts the $k$th row multiplied by $1/2$ from the $(d+k)$th one,
  $k=1$, $\ldots$, $d$. Hence the $(d+1)$st row equals

  $$\begin{array}{lrrrrrrrrrr}
    h:=(~0&-a_{d-1}&0&-a_{d-3}&0&\ldots &-a_{d-2[(d+1)/2]+1}&0&0&\ldots &0~)
  \end{array}$$
  and the next $d-1$ rows are of the form $h_j$, $j=1$, $\ldots$, $d-1$. For
  $d=2$ and $d=3$, this gives

$$\left| \begin{array}{rrrr}
    2&0&2a_0&0\\ 0&2&0&a_0\\ 0&-a_1&0&0\\ 0&0&-a_1&0
  \end{array}\right|~~~\, {\rm and}~~~\, \left| \begin{array}{rrrrrr}
    2&0&2a_1&0&0&0\\ 0&2&0&2a_1&0&0\\ 0&0&2&0&2a_1&0\\
    0&-a_2&0&-a_0&0&0\\ 0&0&-a_2&0&-a_0&0\\ 0&0&0&-a_2&0&-a_0
    \end{array}\right| ~.$$
  \vspace{1mm}

  C) We permute the rows of the determinant (which does not change
  the determinant up to a sign). In the first $d-[d/2]$ positions we place the
  first, third, fifth etc. rows, in the next $[d/2]$ positions
  the $(d+2)$nd, $(d+4)$th, $(d+6)$th etc. rows, in the next $[d/2]$ positions
  the second, fourth, sixth etc. rows and in the last $d-[d/2]$ positions
  the $(d+1)$st, $(d+3)$rd, $(d+5)$th etc. rows. After this permutation
  the first $d$ rows have non-zero entries only in the odd and the last
  $d$ rows have non-zero entries only in the even columns.

  Then we permute the columns of the determinant placing the odd
  columns in
  the first $d$ positions and the even columns in the last $d$ positions by
  preserving the relative order of the even and odd columns. For
  $d=2$ and $d=3$, the result is 

$$\left| \begin{array}{rrrr}
    2&2a_0&0&0\\ 0&-a_1&0&0\\ 0&0&2&2a_0\\ 0&0&-a_1&0
  \end{array}\right|~~~\, {\rm and}~~~\, \left| \begin{array}{rrrrrr}
    2&2a_1&0&0&0&0\\ 0&2&2a_1&0&0&0\\ 0&-a_2&-a_0&0&0&0\\ 0&0&0&2&2a_1&0\\
    0&0&0&-a_2&-a_0&0\\ 0&0&0&0&-a_2&-a_0
  \end{array}\right| ~.$$
  For any $d\geq 2$, the determinant is now block-diagonal, with two
  diagonal blocks $d\times d$. For $d=4$, these blocks are

  $$\left| \begin{array}{rrrr}
    2&2a_2&2a_0&0\\ 0&2&2a_2&2a_0\\
  0&-a_3&-a_1&0\\ 0&0&-a_3&-a_1\end{array}\right| ~~~\, {\rm and}~~~\,
  \left| \begin{array}{rrrr}
     2&2a_2&2a_0&0\\ 0&2&2a_2&2a_0\\ -a_3&-a_1&0&0\\ 0&-a_3&-a_1&0
  \end{array}\right| ~.$$
  The first and the $(d+1)$st rows equal respectively

  $$\begin{array}{lrrrrrrrrrr}
    \tilde{g}:=(~2&&2a_{d-2}&&2a_{d-4}&\ldots &2a_{d-2[d/2]}&0&0&\ldots &0~)
  \end{array}$$
  and $\tilde{g}_d$. The first $d-[d/2]$ rows equal
  $\tilde{g}$, $\tilde{g}_1$, $\tilde{g}_2$, $\ldots$, $\tilde{g}_{d-[d/2]-1}$
  while the rows with indices $d+1$, $d+2$, $\ldots$, $d+[d/2]$ are
  $\tilde{g}_d$, $\tilde{g}_{d+1}$, $\ldots$, $\tilde{g}_{d+[d/2]-1}$. The
  $(d-[d/2]+1)$st row equals

  $$\begin{array}{lrrrrrrrrrr}
    \tilde{h}:=(~0&-a_{d-1}&-a_{d-3}&-a_{d-5}&\ldots &-a_{d-2[(d+1)/2]+1}&0&0&
    \ldots &0~)~.\end{array}$$
  The next $[d/2]-1$ rows are $\tilde{h}_j$, $j=1$, $\ldots$, $[d/2]-1$.
  The last $d-[d/2]$ rows equal $\tilde{h}_k$, $k=d-1$, $\ldots$,
  $2d-[d/2]-2$.

  The total number of transpositions of rows and columns is even, so the sign
  of the determinant does not change.
  \vspace{1mm}
  
  D) One develops the determinant thus obtained w.r.t. its first and then
  w.r.t. its last column. For $d$ even (resp. for $d$ odd), this yields
  $-4a_0\Delta$ (resp. $-2a_0\Delta$), where the
  $(2d-2)\times (2d-2)$-determinant $\Delta$ is block-diagonal, with
  two diagonal blocks $(d-1)\times (d-1)$ each of which is the Sylvester
  matrix of the polynomials $2Q^1$ and $-Q^2$. This implies part~(1)
  of the theorem.
\end{proof}

\begin{proof}[Proof of part (2)]
One can assign quasi-homogeneous weights
    to the variables $a_j$ as follows: $0$ to $a_{d-1}$, $1$ to $a_{d-2}$ and
    $a_{d-3}$, $2$ to $a_{d-4}$ and $a_{d-5}$, $3$ to $a_{d-6}$ and $a_{d-7}$ etc.,
    in accordance with the fact that $a_{d-2}$, $a_{d-4}$, $\ldots$ and
    $a_{d-3}/a_{d-1}$, $a_{d-5}/a_{d-1}$, $\ldots$ are up to a sign elementary
    symmetric polynomials of the roots of $Q^1$ and $Q^2$. Hence $R_0$ is a
    quasi-homogeneous polynomial of weight $d_0:=[(d-1)/2][d/2]$.
    For $d$ even (resp. for $d$ odd), it
    contains monomials $\alpha a_0^{[(d-1)/2]}a_{d-1}^{[d/2]}$ and $\beta a_1^{[d/2]}$,
    $\alpha \neq 0\neq \beta$
    (resp. $\gamma a_1^{[(d-1)/2]}a_{d-1}^{[d/2]}$ and $\delta a_0^{[d/2]}$,
    $\gamma \neq 0\neq \delta$), all other monomials containing factors
    $a_0^k$ and $a_1^s$ only with $k<[(d-1)/2]$ and $s<[d/2]$ (resp.
    with $k<[d/2]$ and $s<[(d-1)/2]$). Hence $R_0$
    cannot be the product of two quasi-homogeneous polynomials
    of weights $b_1$ and $b_2$, $0<b_1,b_2<d_0$.

  \end{proof}

\end{document}